\documentclass{article}

\usepackage{amsmath}
\usepackage{amssymb}
\usepackage{amsthm}
\usepackage[latin1]{inputenc}
\usepackage[english]{babel}
\usepackage[T1]{fontenc}

\usepackage{xcolor}
\usepackage{comment}

\usepackage{latexsym,color,graphicx,enumerate}
\usepackage{hyperref}

\newtheorem{thm}{Theorem}[section]
\newtheorem{prop}[thm]{Proposition}
\newtheorem{coro}[thm]{Corollary}
\newtheorem{lemma}[thm]{Lemma}
\newtheorem{rem}[thm]{Remark}
\newtheorem{hypo}{Hypothesis}[section]

\makeatletter

\newcommand*{\plim}[1][]{%
	\if\relax\detokenize{#1}\relax
	\def\next{\qopname\relax m{lim}}%
	\else
	\def\next{\qopname\newmcodes@ m{#1-lim}}%
	\fi
	\next
}

\newcommand*{\psum}[1][]{%
	\DOTSB
	\if\relax\detokenize{#1}\relax\else
	\operatorname{#1-}\mkern-\thinmuskip
	\fi
	\sum@\slimits@
}

\makeatother

\DeclareMathOperator\supp{supp}
\newcommand{\R}{\mathbb{R}}             
\newcommand{\N}{\mathbb{N}}             
\newcommand{\Z}{\mathbb{Z}}             
\newcommand{\C}{\mathbb{C}}             

\newcommand{\half}{\frac{1}{2}}

\newcommand {\clb}{\color{blue}}
\newcommand {\clr} {\color{red}}

\newcommand{\n}{n-\nu}

\numberwithin{equation}{section}

\title{On the magnetic Dirichlet to Neumann operator on the  exterior of the disk -  diamagnetism, weak-magnetic field limit and flux effects}

\author{Bernard Helffer  and Fran{\c{c}}ois Nicoleau}

\begin{document}

\maketitle

\begin{abstract}
In this paper, we analyze the magnetic Dirichlet-to-Neumann operator (D-to-N map)  $\check \Lambda(b,\nu)$ on the exterior of the disk with respect to a magnetic potential $A_{b, \nu}=A^b + A_\nu$ where, for $b\in \mathbb R$ and $\nu \in \mathbb R$, 
$A^b (x,y)= b\,  (-y, x)$ and $A_\nu (x,y)$ is  the Aharonov-Bohm potential centered at the origin of flux $2\pi \nu$.  First,  we show that the limit of $\check \Lambda(b,\nu)$ as $b\rightarrow 0$ is equal to the D-to-N map $\widehat \Lambda (\nu)$ on the interior of the disk associated with the potential $A_\nu (x,y)$. Secondly, we study the ground state energy of the D-to-N map $\check \Lambda(b,\nu)$ and show that the strong diamagnetism property holds.  Finally we slightly extend to the exterior case the asymptotic results obtained in the interior case for general domains.
 \end{abstract}

\newpage 
\tableofcontents



\section{Introduction}

In this paper, we study the weak-field limit and the ground state energy of the magnetic Dirichlet to Neumann operator\footnote{See \cite{GP} for a general introduction.}  (in what follows D-to-N map) in the case  of a magnetic potential with a constant  associated magnetic field. When $\Omega$
 is not simply connected  (which is the case of the exterior of the disk), this leads us to consider a family of potentials  with  one potential of the family different to the  other by an  Aharonov-Bohm (A-B)  potential 
$A_\nu (x,y)$, centered in the complementary of $\overline{\Omega}$ which is assumed to be connected. For  the exterior on the unit disk $D(0,1) \subset \R^2_{x,y}$ we consider the A-B potentials centered at $0$.

\vspace{0.1cm}\noindent
First, let us recall some basic facts on the magnetic D-to-N map in the case of a bounded domain  $\Omega \subset \R^2$, with smooth boundary. For any $ u \in \mathcal D'(\Omega)$, 
the magnetic Schr\"odinger operator  on $\Omega$ is defined as 
\begin{equation}\label{defMagOp}
	H_{A}\  u = (D-A)^2 u = -\Delta u +2i \   A \cdot \nabla u  + (A^2 + i \  {\rm{div}} \ A ) u,
\end{equation}
where   $D= -i \nabla$,  $-\Delta$ is the usual positive Laplace operator on $\R^2$ and $A = (A_1, A_2)$ is the magnetic potential vector field.  When considered as a $1$-form, we write  $\omega_A=  A_1\ dx  + A_2\ dy$ 
 and the magnetic field is given by the $2$-form $\sigma_B :=d\omega_A$.  We  write $\sigma_B = B \, dx\wedge dy$, which permits to identify the $2$-form $\sigma_B$ with the function $B$.

\vspace{0.1cm}\noindent
The boundary value  problem 
\begin{equation}  \label{Dirichletdisque}
	\left\{
	\begin{array}{rll}
		H_{A} \  u &=&0  \  \ \rm{in}  \ \ \Omega,\\
		u_{  \vert \partial \Omega} & =& f \in H^{1/2}(\partial\Omega) ,
	\end{array}\right.
\end{equation}
has a unique solution $u \in H^1(\Omega)$ since zero does not belong to the spectrum of the Dirichlet realization  of  $H_{A}$.   Then, the  D-to-N map, is defined  by
\begin{equation} \label{D-to-N--map}
	\begin{array}{rll}
		\Lambda_{A} :   H^{1/2}(\partial\Omega)& \longmapsto & H^{-1/2}(\partial\Omega) \\ 
		f   &\longmapsto&  \left(\partial_{\vec{\nu}} \, u + i \langle A, \vec{\nu} \rangle \ u  \right)_{  \vert \partial \Omega} ,
	\end{array}
\end{equation}
where $\vec{\nu}$ is the outward normal unit vector field on $\partial\Omega$.  More precisely, we define the D-to-N map using the equivalent weak formulation: 
\begin{equation}\label{DtNweak}
	\left\langle \Lambda_{A} f , g \right \rangle_{H^{-1/2}(\partial \Omega) \times H^{1/2}(\partial \Omega)} = \int_\Omega  \langle (D-A)u , (D-A)v \rangle\ dx\,,
\end{equation}
for any $g \in H^{1/2}(\partial \Omega)$ and $f \in H^{1/2}(\partial \Omega)$ such that $u$ is the unique solution of (\ref{Dirichletdisque}) and $v$ is any element of $H^1(\Omega)$ 
so that $v_{|\partial \Omega} = g$. Clearly, the D-to-N map is a positive operator.

\vspace{0.1cm}\noindent
We recall (see \cite{GP,HKN} and references therein) that the spectrum of the D-to-N operator is discrete and is given by an increasing sequence of eigenvalues 
\begin{equation} \label{spectrum}
	0 \leq 	\mu_1 \leq \mu_2 \leq ... \leq \mu_n \leq ...  \to + \infty\,. 
\end{equation}

\vspace{0.3cm}
\noindent
When $\Omega$ is an unbounded open set, the definition of the magnetic D-to-N map  is not quite as simple as in the case of bounded domains. As far as we know, for compactly supported magnetic fields,  the  D-to-N operator in the half-space $\R^3_+$  was well defined  in (\cite{Li}, Appendix B) using the Lax-Phillips method, and  also on an infinite slab $\Sigma$ in \cite{KrUh}. At last, for non-compactly supported electromagnetic fields, the  D-to-N map on an unbounded open set $\Omega \subset \R^3$ corresponding to a closed waveguide was studied in \cite{Ki}.  We should describe further the definition in the next section  in the case of a non vanishing magnetic field at infinity.

\vspace{0.2cm}\noindent
In this paper, we consider the following family of magnetic $1$-forms  in the exterior of the unit disk $\Omega = R^2\setminus D(0,1)$  defined as:
\begin{equation}\label{constantbAB}
	A_{b, \nu}(x,y) = A^b(x,y) + 	A_\nu (x,y) \,.
\end{equation}
The first magnetic potential appearing in the (RHS) of (\ref{constantbAB}) is 
\begin{equation}
A^b(x,y) = b\, (-y, x),
\end{equation}
 where $b$ is a fixed real constant. The associated magnetic field is constant  of strength $2b$. \\ The second magnetic potential is the  so-called magnetic Aharonov-Bohm-potential  defined as:
\begin{equation}\label{defpotAB}
	A_{\nu} (x,y) = \frac{\nu}{r^2} (-y,x)\quad , \quad  \nu \in \mathbb R\,.
\end{equation}
This magnetic potential $	A_\nu$ creates a flux $2\pi \nu$ around the origin, and in the distributional sense we have:

\begin{equation}
	{\rm curl}\,A_\nu = 2\nu \ \delta_0\,.
\end{equation}

Clearly, the magnetic two-form $d \omega_{ A_{\nu,b}}$ satisfies
\begin{equation}
 d \omega_{ A_{b, \nu} }= 2b\ dx\wedge dy  \mbox{ in }\Omega\,.
\end{equation}
Note also that the magnetic potential $ A_{b,\nu}$ satisfies the Coulomb gauge  condition  
 \begin{equation}
 {\rm div} A_{b, \nu}=0 \mbox{  and } \langle  A_{b, \nu}, \vec{\nu} \rangle =0 \mbox{ on } S^{1}, \mbox{ (the boundary of }\Omega)\,.
\end{equation}
At last, thanks to a natural gauge invariance, we can always assume that $\nu \in (-\half, \half]$.



\vspace{0.2cm}\noindent
Note that the  D-to-N map $\widehat \Lambda(\nu)$ associated with the magnetic potential $A_\nu(x,y)$ in the interior of the unit disk $D(0,1)$ was computed explicitly in \cite{CPS}. In particular, the authors show that the spectrum of $\widehat \Lambda(\nu)$ is given by:
\begin{equation}
	\lambda_k (\nu) =|k-\nu | \ \mbox{ for } \ k\in \mathbb Z \,.
\end{equation}
Except for $\nu = \frac 12$, each eigenvalue has multiplicity $1$.

\vspace{0.3cm}
The first goal of this paper is to rigorously define, when $b\neq 0$,  the D-to-N map  $\check\Lambda(b, \nu)$ associated with the potential $ A_{b, \nu} (x,y)$ on $\Omega$, (i.e, in the exterior of $D(0,1)$). This will actually be defined  in the next section  for a rather general unbounded domains with bounded boundary and rather general non vanishing magnetic fields. We shall denote $\check \lambda^{DN} (b,\nu)$ its ground state energy. 

\vspace{0.2cm}\noindent
Then, in parallel with the analysis by Chaigneau and Grebenkov \cite{CG} of the limit as $p\rightarrow 0$ of the D-to-N problem associated with $-\Delta +p$ on the exterior of a compact set, we study the weak-field limit (i.e. the limit as $b\rightarrow 0$) of the D-to-N map  $\check \Lambda(b, \nu)$ and  get:

\begin{thm}\label{weaklimitintro}
	For any $b>0$ and any  $\nu \in (-\half,\half]$,  $\check \Lambda(b,\nu)-\widehat \Lambda(\nu) \in {\cal{B}}(L^2(S^{1}))$. Moreover, we have as 
	$b \to 0^+\,$: 
	\begin{eqnarray}
		|| \check \Lambda(b,\nu)- \widehat \Lambda(\nu)||_{ {\cal{B}}(L^2(S^{1})) } &=& \mathcal O(\frac{1}{|\log b|}) \quad  ,  \quad {{\rm if}} \ \nu =0, \\
		 &=& \mathcal O(b^{|\nu|}) \quad \quad \  , \quad {{\rm if}} \ \nu \not=0.
	\end{eqnarray}
\end{thm}


\vspace{0.2cm}\noindent
Actually,  when $\nu=0$,  we can get a more accurate asymptotic estimate for the difference  $\check \lambda_n (b)- |n|$, 
where,  for $n \in \Z$, $\check \lambda_n (b)$ denote the eigenvalues of the D-to-N map  $\check\Lambda(b, 0)$:


\begin{prop}
When $b \to 0^+$, one has :	
\begin{eqnarray}
     \check \lambda_0(b) &=& -\frac{2}{\log b} + \mathcal O(\frac{1}{(\log b)^2}) , \\
	\check \lambda_n (b) -|n|  &=& b \log b  + \mathcal O(b) \ ,\ |n| =1 ,\\ 
	\check \lambda_n (b) - |n| &=& - \frac{n}{|n|-1} \ b + \mathcal O(b^2) \ ,\ |n| \geq 2\,.
\end{eqnarray}

\end{prop}

\begin{rem}
In the same way, when $\nu \not=0$, we can theoretically estimate the difference $\check \lambda_n (b,\nu)- |n-\nu|$   
where  $\check \lambda_n (b, \nu )$ denote the eigenvalues of the D-to-N map  $\check\Lambda(b, \nu)$, but this leads to very cumbersome computations even with the help of a computer. Nevertheless, in the particular case  $n=0$, one can get:  
\begin{equation} 
	\check \lambda_0(b,\nu) - |\nu| =  \frac{2 \ \Gamma(1-|\nu|) \ \Gamma(|\nu|+\half)}{\sqrt{\pi} \ \Gamma (|\nu|)} \  b^{|\nu|} +  \mathcal O(b^{2|\nu|})\,.
\end{equation}
\end{rem}


\vspace{0.2cm}\noindent
Our second result is concerned by strong diamagnetism. We recall that by diamagnetism we mean that $\check \lambda^{DN} (b,\nu) $ is minimal for $b=0$. This result has been proved in full generality in \cite{EO}. 

\vspace{0.2cm}\noindent
In the particular case of the exterior of the disk
we  prove (the analogous result for the disk was proven in \cite{HN}) the stronger result:

\begin{thm}
	For any fixed $\nu \in (-\half,\half]$, the map $b \mapsto \check \lambda^{DN} (b,\nu) $ is increasing on $(0,+\infty)$. 
\end{thm}

\vspace{0.1cm}\noindent
Finally, in continuation of our work \cite{HN}, we study the ground state energy of the D-to-N map $\check \Lambda(b,\nu)$
and  extend to $\nu \in (-\half,\half]$ the result of \cite{HKN} (which was proven when $\nu=0$ with another approach).
\begin{thm}\label{main}
	For any fixed $\nu \in (-\half,\half]$, one has the asymptotic expansion as $b \to + \infty$,
	\begin{equation}
	\check \lambda^{DN} (b,\nu) = \alpha b^{1/2}  + \frac{ \alpha^2 +2}{6} + \mathcal O_\nu (b^{-1/2})\,,
	\end{equation}
	where $-\alpha$ is the unique negative zero  of the so-called parabolic cylinder function $D_{\half} (z)$ .
\end{thm}

\vspace{0.2cm}\noindent
We recall that the parabolic cylinder functions $D_\mu (z)$ are the (normalized)  solutions of the differential equation 
\begin{equation}\label{ODEDnu0}
	w'' + (\mu  + \frac 12 - \frac{z^2}{4})\ w=0\,,
\end{equation}
which tend to $0$ as $z \to +\infty$. We refer to Section \ref{champconstant} for more details on the parabolic cylinder functions. 

\vspace{0.2cm}\noindent
At last, the positive real $\alpha$ appearing in this theorem is approximately equal to 
\begin{equation}\label{approalpha}
	\alpha = 0.7649508673 ....
\end{equation}

It is worth to notice that the two first terms of the expansion are independent of $\nu$ as it was predicted by the proofs in \cite{HKN} for general bounded domains. We refer to the last section for a more developed discussion.


\vspace{0.3cm}
\noindent \textbf{Acknowledgements}:  One motivation of this paper comes from a question of Vincent Bruneau, 
when the first author was presenting the results of \cite{HN}  in the case of the disk in Bordeaux. We are also  grateful to Denis Grebenkov, Ayman Kachmar  and Mikael P. Sundqvist for helpful discussions and remarks around  this work.\\


\section{Generalities for the exterior problem with variable magnetic field}
\vspace{0.1cm}\noindent
The case of the exterior problem has been analyzed in the magnetic field case under various assumptions (see for example \cite{GKPS,FPS0,KLPS}). Here we will work under the following assumption:
\begin{hypo}\label{eq:H00}
We assume that  $\Omega$ is the complementary of a bounded regular connected set in $\mathbb R^2$ and that
\begin{equation}
\liminf_{x\in \Omega, |x|\rightarrow+\infty} B(x) >0\,.
\end{equation}
\end{hypo}
Under this condition, using a variant of  the results of \cite{HeMo, KPS} based on Persson's Lemma, the Dirichlet magnetic Laplacian $H_A^D$ has its essential spectrum  bounded from below:
\begin{equation}\label{eq:essspD}
\inf \sigma_{ess}(H_A^D) \geq \liminf_{x\in \Omega, |x|\rightarrow+\infty} B(x)\,,
\end{equation}
and the same holds for the Neumann realization
\begin{equation}\label{eq:essspN}
\inf \sigma_{ess}(H_A^{Ne}) \geq \liminf_{x\in \Omega, |x|\rightarrow+\infty} B(x)\,.
\end{equation}

Let us show 
\begin{lemma}
Under Hypothesis \ref{eq:H00} and assuming that $\Omega$ is connected, we have
\begin{equation}
\inf \sigma(H_A^D) >0\,,
\end{equation}
and
\begin{equation}
\inf \sigma(H_A^{Ne}) >0\,.
\end{equation}
\end{lemma}
\begin{proof}
The proof is by contradiction. By the variational characterization of the spectrum, it is enough to consider the Neumann case. If $0$ was in the spectrum, it would be by \eqref{eq:essspN}  an eigenvalue. Hence there would be a non zero  eigenfunction $u$ such that
\begin{equation}\label{differential}
du = i \omega_A u\,,
\end{equation}
where $\omega_A$ is the magnetic potential, considered as a one form.\\
Taking the differential above, we get:
\begin{equation}
0=  \omega_A \wedge du + \sigma_B u \,,
\end{equation}
where we recall that $\sigma_B$ is the magnetic field considered as a $2$-form, i.e  $\sigma_B = B\,  dx\wedge dy$. 
Using again (\ref{differential}), we get immediately:
\begin{equation}
B u =0\,.
\end{equation}
By the condition at $\infty$, $u$ would be $0$ outside a large disk, and by unique continuation theorem\footnote{We recall that, in dimension $n \geq 3$, the unique continuation principle holds for a uniformly elliptic operator on a domain $\Omega$ if the coefficients of the principal part of this operator are locally Lipschitz continuous, whereas in dimension $n=2$, the unique continuation principle holds if the coefficients of the principal part are $L^{\infty}$ (see for instance \cite{Hor1, Hor2, Tat}).}
 for an eigenfunction of $H_A$ and the connectedness of $\Omega$, we get $u=0$. 

\vspace{0.1cm}\noindent
For completion, we give a more direct argument which does not involve the unique continuation principle. Let $x_0 \in \partial \supp u \cap \Omega$ and $v\in \mathbb R^2$ of norm $1$. Let us consider the function $t \mapsto \phi_{x_0,v}(t):= u(x_0+t v)$ which is well defined for $|t| < d(x_0,\partial \Omega)$. We now notice that $\phi_{x_0,v}$ is the solution of the ordinary differential equation
$$
\frac{d}{dt} \phi_{x_0,v}(t) = i < A(x_0+vt) , v>  \phi_{x_0,v}(t)
$$
with
$$
\phi_{x_0,v}(0)=0\,.
$$
This implies by the Cauchy-Lipschitz theorem that $u$ vanishes in the disk centered at $x_0$ of radius $d(x_0,\partial \Omega)$ in contradiction of the choice of $x_0$.
\end{proof}

Since  $\partial \Omega$ is compact, the construction given in the bounded case works identically in the exterior case if we replace $H^1(\Omega)$ by the magnetic  Sobolev space $$H^1_A:=\{u\in L^2(\Omega)\,,\, (D-A) u\in L^2(\Omega;\mathbb R^2)\}\,.
$$ 
Notice that the trace space at $\partial \Omega$ is the same and independent of the magnetic potential:
$$
H^{1/2}_A(\partial \Omega)= H^{1/2}(\partial \Omega)\,,
$$
and that the situation could be more delicate if $\partial \Omega$ is not compact (see \cite{HKN} for a discussion).

\begin{prop}
Under Hypothesis \ref{eq:H00} and assuming that $ \Omega$ is connected  we have
$$
\inf \sigma(\Lambda_A) >0\,,
$$
where $\Lambda_A$ denote the D-to-N map associated with the magnetic potential $A$ on $\Omega$.
\end{prop}

\begin{proof}
Notice that 
$$
\left\langle \Lambda_{A} f , f \right \rangle_{H^{-1/2}(\partial \Omega) \times H^{1/2}(\partial \Omega)} \geq \inf \sigma(H_A^{Ne}) \ ||u||^2\,.
$$
Together with the continuity of the map $u\mapsto u_{|\partial \Omega}$ from $H^1_A(\Omega)$ onto $H^{\frac 12}(\partial \Omega)$, this implies that 
 the D-to-N map is a positive operator.
 \end{proof}

All these conditions are satisfied for the case of the complementary of the disk and the non-zero constant magnetic case, which will be our main interest in this paper.

\section{ Constant magnetic field in the disk--Reminder} \label{champconstant}

In this section, we recall the main results of \cite{HN} in the case of the disk and for the magnetic potential $A^b(x,y) = b\,  (-y,x)$ only, (i.e in absence of the (A-B) potential $A_\nu (x,y)$).  Like initiated in \cite{S-J} this approach is based on the use of 
 families of special functions (see \cite{SoSo,BW24,HeLe,KMi} for other applications in the same spirit).

\vspace{0.1cm}\noindent
We use the standard Fourier decomposition to solve the boundary value problem (\ref{Dirichletdisque}). In polar coordinates $(r, \theta)$, the D-to-N map is defined (in a weak sense) by:
\begin{equation}\label{defD-to-N--}
	\begin{array}{lll}
		\Lambda^{\rm DN}(b) : H^{\half} (S^1) & \to &    H^{-\half} (S^1) \\
		\hspace{1.5cm} \Psi  &\to& \partial_r v (r, \theta)|_{r=1} \,,
	\end{array}
\end{equation}
where $v$ is the solution of \eqref{Dirichletdisque} with $f=\Psi$ expressed in polar coordinates.\\

\noindent
We write the solution $v(r, \theta)$ in the form  
\begin{equation}
	v(r, \theta) = \sum_{n \in \Z} v_n (r) e^{in \theta}\ \ , \ \  \Psi(\theta) = \sum_{n \in \Z} \Psi_n e^{in \theta}, 
\end{equation}
and we see, (\cite{CLPS1}, Appendix B), that $v_n (r)$ solves the ODE:
\begin{equation}\label{polarequations}
	\left\{
	\begin{array}{ll}
		- v_n'' (r) - \frac{v_n'(r)}{r} + (br-\frac{n}{r} )^2 v_n (r)= 0   & \mbox{for} \  r \in (0,1) , \\
		v_n (1) = \Psi_n .
	\end{array}
	\right.
\end{equation}
A bounded solution to the differential equation (\ref{polarequations}) is given (see \cite{CLPS1}, Eq. (B.2))  by 
\begin{equation}\label{solutionpos}
	v_n(r) = c_n e^{-\frac{br^2}{2}} r^nM(\half, n+1, br^2)   \ \ \mbox{for} \ n \geq 0\, ,
\end{equation}
where $M(a,c,z)$ is the Kummer confluent hypergeometric function, (this function is also denoted by $_1F_1 (a,c,z)$ in the literature), and is defined as 
\begin{equation}\label{Kummerfunction}
	M(a,c,z) = \sum_{n=0}^{+\infty} \frac{(a)_n}{(c)_n} \ \frac{z^n}{n!}\,.
\end{equation}
\vspace{0.1cm}\noindent
Here $z$ is a complex variable, $a$ and $c$ are parameters which can take arbitrary real or complex values, except that $c \notin \Z^-$. At last, 
\begin{equation}
	(a)_0 =1 \ ,\ (a)_n = \frac{\Gamma(a+n)}{\Gamma(a)} = a(a+1) ...(a+k-1), \ k=1,2, ... \,,
\end{equation}
are the so-called Pochhammer's symbols, (see \cite{MOS1966}, p. 262). Finally, the  function $M(a,c,z)$ satisfies the differential equation:
\begin{equation}\label{KummerODE}
	z\frac{d^2 w}{dz^2}  + (c-z)\frac{dw}{dz} -a w=0\,.
\end{equation}
Note that for $n \leq -1$ and thanks to symmetries in  (\ref{polarequations}), we get a similar expression for $v_n(r)$  
changing the parameters $(n,b)$ into $(-n, -b)$.

\vspace{0.1cm}\noindent
Now, let us return to the study of the eigenvalues of the D-to-N map $\Lambda(b)$. They are usually called {\it magnetic Steklov eigenvalues} and  given by
\begin{equation}
	\lambda_n (b)= \frac {v_n' (1)}{v_n (1)} \ \ \mbox{for} \ n \in \Z \,.
\end{equation}
Thus, using (\ref{solutionpos}),  we see that the {\it magnetic Steklov spectrum} is the set:
\begin{equation}\label{spectrumSpos}
	\sigma(\Lambda^{\rm DN}(b)) = \{\lambda_0(b)\} \cup \{\lambda_n(b) , \lambda_n(-b) \ \}_{n \in \N^*},
\end{equation}
where for $n \geq 0$,
\begin{equation}\label{explicitvp}
	\lambda_n (b) =  n - b +  2b \ \frac{M'(\half, n+1, b)}{M(\half, n+1,b)}.
\end{equation}

\vspace{0.2cm}\noindent
In \cite{HN}, motivated by questions in (\cite{CGHP}, Example 2.8), we were interested in the analysis of
\begin{equation}
	\lambda^{\rm DN}(b) :=\inf_{n\in \Z} \lambda_n(b)\,,
\end{equation}
as $b\rightarrow +\infty$.

\begin{figure}
	\begin{center}
			\includegraphics[width=0.48\textwidth]{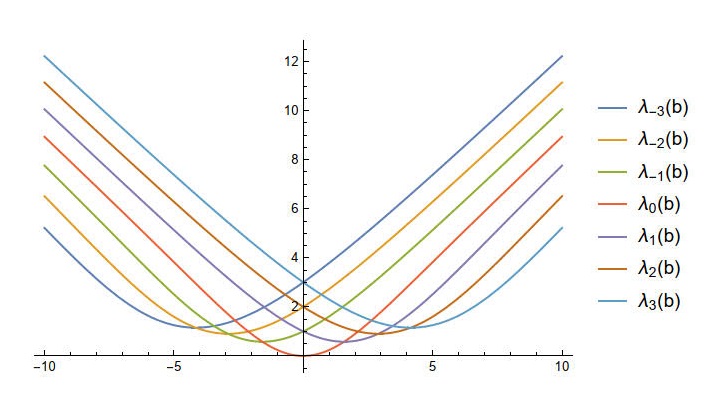}
		\includegraphics[width=0.48\textwidth]{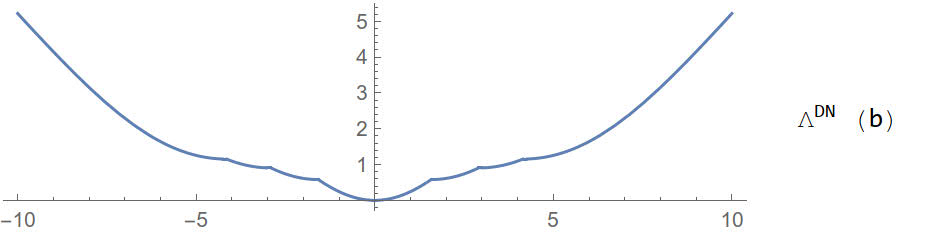} 
			\end{center}
	\caption{The Steklov eigenvalues $\lambda_n (b)$ (left) on the disk and the ground state energy $\lambda^{DN}(b) $ (right). }
\end{figure}

\vspace{0.2cm}\noindent
The first result of \cite{HN} is the following:
\begin{thm}\label{mainHN}
	One has the asymptotic expansion as $b \to + \infty$,
	\begin{equation}
		\lambda^{DN}(b) = \alpha \,  b^{1/2} - \frac{\alpha^2 +2}{6} + \mathcal O (b^{-1/2})\,,
	\end{equation}
	where $-\alpha$ is the unique negative zero  of the so-called parabolic cylinder function $D_{\half} (z)$ .
\end{thm}

\vspace{0.2cm}\noindent
We recall that the parabolic cylinder functions $D_\mu (z)$ are solutions of the differential equation  \eqref{ODEDnu0} which tend to $0$ as $z \to +\infty$. 
For any $\mu <0$, one has the following integral representation (\cite{MOS1966}, p. 328):
\begin{equation}\label{integralrep}
	D_\mu (z) = \frac{e^{- \frac{z^2}{4}}}{\Gamma(-\mu)} \ \int_0^{+\infty} t^{-\mu -1} e^{- ( \frac{t^2}{2} +zt)} \ dt \,.
\end{equation} 
The parabolic cylinder functions have the following asymptotic expansion (\cite{MOS1966}, p. 331):
\begin{equation}\label{asymptDnu}
	D_\mu (z) = e^{- \frac{z^2}{4} } z^\mu \ (1+ \mathcal O (\frac{1}{z^2})) \ ,\ z \to + \infty.
\end{equation}
Notice that for $\mu <0$, these asymptotics are obtained by applying the Laplace integral method in \eqref{integralrep}.
The parabolic cylinder functions $D_\mu (z)$ satisfy the recurrence relations  (\cite{MOS1966}, p. 327),
\begin{subequations}
	\begin{equation}\label{recurrenceDnu}
		D'_\mu (z) - \frac{z}{2} D_\mu (z) + D_{\mu+1}(z)=0 \,,
	\end{equation}
	\begin{equation}\label{recurrenceDnu1}
		D_{\mu+1} (z) - z D_\mu (z) + \mu D_{\mu-1}(z)=0 \,.
	\end{equation}
	\begin{equation}\label{recurrenceDnu2}
		D'_{\mu} (z) + \frac{z}{2} D_\mu (z) -\mu  D_{\mu-1}(z)=0 \,,
	\end{equation}
\end{subequations}

\vspace{0.2cm}\noindent
At last, the second result obtained in  \cite{HN} is concerned by strong diamagnetism.

\begin{thm}
	The map $b \mapsto \lambda^{DN}(b) $ is increasing on $(0,+\infty)$. 
\end{thm}

\vspace{0.2cm}\noindent
 Notice that the case of the Neumann realization in the interior of the disk has been analyzed extensively (see \cite{HeLe} and references therein) and the case in the exterior of the disk is analyzed in \cite{GKPS,KLPS,HKN}. The Dirichlet case
  is also analyzed in \cite{BW24}.

\section{Magnetic Steklov eigenvalues on the exterior of the disk.}\label{s4}

In this section, we assume that the magnetic potential is given by $ A(x,y) = b (-y, x)$, (i.e we assume that the flux $\nu=0$). The case $\nu \in  (\half,\half]\backslash \{ 0 \}$ will be studied in Section \ref{ABsection}.

\subsection{Special functions}
Similarly to the interior case, we consider the following ordinary differential equations where $n \in \Z$:
\begin{equation}\label{polarequationsext}
	\left\{
	\begin{array}{ll}
		- v_n'' (r) - \frac{v_n'(r)}{r} + (br-\frac{n}{r} )^2 v_n (r)= 0   & \mbox{for} \  r \in (1,+\infty) , \\
		v_n (1) = 1\,.
	\end{array}
	\right.
\end{equation}
We emphasize that this time (in comparison with \eqref{polarequations}), the interval is $(1,+\infty)$ instead of $(0,1)$.

\vspace{0.2cm}\noindent
For a magnetic field $b \geq 0$, the unique bounded solution at infinity $v_n$, (actually $v_n$ decays exponentially) is given by 
\begin{equation}
	v_n (r)= c_n e^{- b r^2/2} r^n \ U (\frac 12,n+1, br^2)\,,
\end{equation}
where $c_n$ is a suitable constant.\\ We also note that, instead of the Kummer function $M(a,c,z)$ introduced  in Section \ref{champconstant}, 
we introduce a new function denoted by $U(a,c,z)$. This function is called {\it{the confluent hypergeometric  function of the second kind}}. As we will see later, $U(a,c,z)$ is  better adapted to the study of our  exterior problem.
 
\vspace{0.2cm}
\par\noindent
First, we observe that, although the function $M(a,c,z)$ is undefined if $c=-m$ with $m \in \N$, the following limit exists (\cite{MOS1966}, p. 263):
\begin{equation}\label{prolongeM} 
\lim_{c \to -m} \ \frac{1}{\Gamma(c)} M(a,c,z) = \frac{(a)_{m+1}} {(m+1)!} \ z^{m+1}\  M(a+m+1,m+2,z)\,.
\end{equation}
Thus, we can define for any $a,c \in \C$ and  $ -\pi<{\rm arg } \ z \leq \pi$,
\begin{equation}\label{defU}
	U(a,c,z) = \frac{\pi}{\sin(\pi z)}\ \left(\frac{ M(a,c,z) }{\Gamma(c)\Gamma(1+a-c)}  - z^{1-c} \frac{  M(1+a-c,2-c,z)}{\Gamma(a)\Gamma(2-c)} \right)\,.
\end{equation}
Actually, we can define $U(a,c,z)$ as a multiple-valued function with its principal branch  given by $-\pi<{\rm arg} \ z \leq \pi$. For more details concerning the analytic continuation of $U(a,c,z)$ on the Riemann surface,  
see \cite{MOS1966}, p. 263.
\vspace{0.2cm}
\par\noindent
For fixed values of $a$ and $c$, the hypergeometric function $U(a,c,z)$  has the following asymptotics as $z \to +\infty$, (\cite{MOS1966}, p. 289): 
\begin{equation}\label{asymptotiqueU}
\forall N\in \N\,,\, 	U(a,c,z) = \sum_{n=0}^N (-1)^n \frac{(a)_n (a+1 -c)_n}{n!} z^{-n-a} + \mathcal O(|z|^{-N-a-1})\,.
\end{equation}
At last, we have  (see \cite{MOS1966}, p. 277)  the following integral representation for $U(a,c,z)$, 
\begin{equation}\label{eq:integralrepU}
U(a,c,z)= \frac{1}{\Gamma(a)} \int_0^{+\infty} e^{-zt} t^{a-1} (1+t)^{c-a-1} dt\,\quad , \quad \Re a >0, \ \Re z >0\,.
\end{equation}
It follows that for $a>0$, $ c \in \R$, the function $z \to U(a,c,z)$ does not have real zeros.
\vspace{0.2cm}
\par\noindent
Similarly with  the derivative of $M(a,c,z)$ (see \cite{HN}), we get for the derivative of $U(a,c,z)$ with respect to $z$ and denoted by $U'(a,c,z)$:
\begin{equation}\label{derivU}
 U'(a,c,z) := - a \ U (a+1,c+1,z)\,.
\end{equation}

\par\noindent
For the convenience of the reader, we recall also  some of the relations (see \cite{MOS1966}, p. 265). \\
For any $a,c,z \in \C$, one has:
\begin{subequations}\label{ContiguousU}
\begin{eqnarray}
	&  U(a,c,z) - U(a,c-1,z) - a U(a+1,c,z)=0 \,.  \\
	& U(a-1,c,z) + (c-a) U(a,c,z) -z U(a,c+1,z)=0 \,.  
\end{eqnarray}
\end{subequations}

\subsection{Steklov eigenvalues}
We now compute the magnetic Steklov eigenvalues $\check \lambda_n(b)$ for this exterior problem. We begin with:
\begin{equation}
v'_n(r) = ( -  br  + \frac{n}{r}) v_n + 2 b r c_n e^{- b r^2/2} r^n   U' (\frac 12,n+1, br^2) 
\end{equation}
This leads 
to
\begin{equation}\label{explicitvpext}
\check \lambda_n(b) :=- \frac{v'_n(1)}{v_n(1)}= -n+b -  2b  \frac{ U' (\frac 12,n+1, b) }{ U (\frac 12,n+1, b) } \,.
\end{equation}
Using (\ref{derivU}), we easily see, to compare with \eqref{explicitvp},  that 
\begin{equation}\label{explicitvpext1}
\check \lambda_n(b) = -n+b +b \ \frac{  U (\frac 32,n+2, b) }{ U (\frac 12,n+1, b) }\,.
\end{equation}
Moreover,  thanks to symmetries in  (\ref{polarequationsext}), for a magnetic field $b \leq 0$, we get immediately: 
\begin{equation}\label{symetries}
\check \lambda_n(b) = \check \lambda_{-n}(-b)\,.
\end{equation}
Thus, we get the following picture for the family of eigenvalues $\check \lambda_n(b)$ associated with the D-to-N operator  $\check \Lambda (b)$, in the case of the exterior of a disk (see Figure 2), as well as its ground state energy
 \begin{equation}
  \check\lambda^{DN}(b):=\inf_{n\in \Z} \check \lambda_n(b)\,. 
  \end{equation}


\begin{figure}
	\begin{center}
		\includegraphics[width=0.48\textwidth]{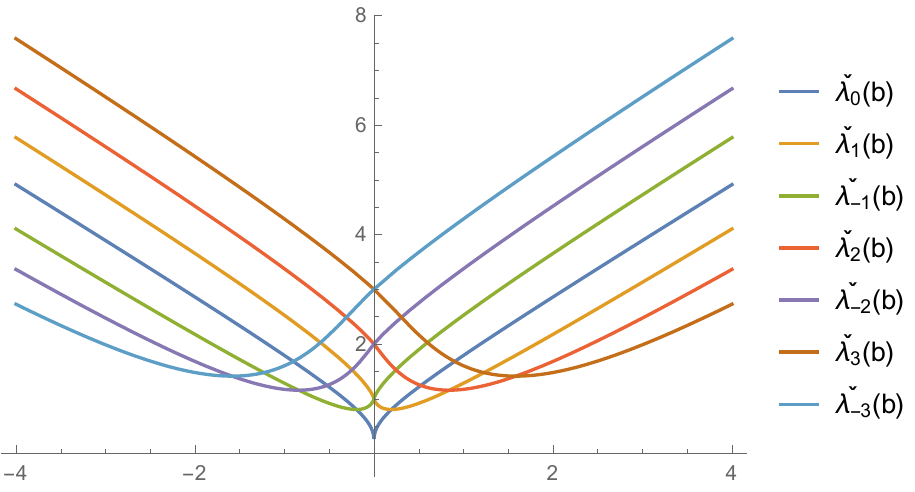}
		\includegraphics[width=0.48\textwidth]{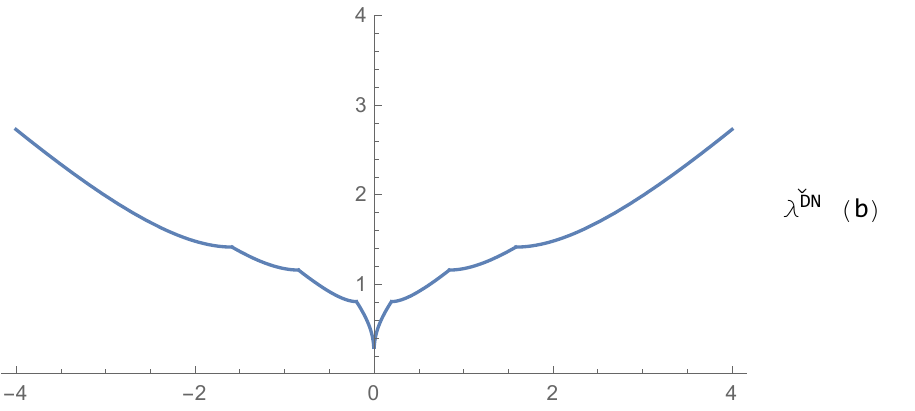} 
	\end{center}
\caption{The magnetic Steklov eigenvalues $\check \lambda_n(b)$ (left) and  the ground state energy $\check \lambda^{DN}(b)$ (right).}
\end{figure}

\newpage


\subsection{Weak magnetic field limit of the exterior D-to-N map $\check\Lambda(b)$.}
\vspace{0.1cm}
\par\noindent
The case $n=0$ is particularly interesting (see Figure 3). Indeed, we recall (\cite{MOS1966}, p. 288 or \cite{SoSo}) that the following hypergeometric functions  satisfy the asymptotic expansions as $z \to 0^+$: 
\begin{subequations}\label{asymptotUzero}
\begin{eqnarray}
	U\left(a,1,z\right)&=&-\frac{1}{\Gamma\left(a\right)}\left(\ln z+\frac{\Gamma'(a)}{\Gamma(a)}-2\gamma\right)+\mathcal O\left(z\ln z\right), \\
	U\left(a,2,z\right)&=&\frac{1}{\Gamma\left(a\right)}z^{-1}+\mathcal O\left(\ln z\right), \\ 
	U\left(a,n,z\right)&=&\frac{\Gamma(n-1)}{\Gamma(a)}z^{1-n}+\mathcal O\left(z^{2-n}\right) \ ,\ n \geq 3.
\end{eqnarray}
\end{subequations}
Using (\ref{explicitvpext1}), we see that the weak-field limit  for the eigenvalue  $\check \lambda_0(b)$ is given by:
\begin{equation} \label{lecasn=0}
	\check \lambda_0(b) = -\frac{2}{\log b} + \mathcal O(\frac{1}{(\log b)^2}) \quad , \quad b \to 0^+ \,.
\end{equation}

\begin{figure}
	\begin{center}
		\includegraphics[width=0.48\textwidth]{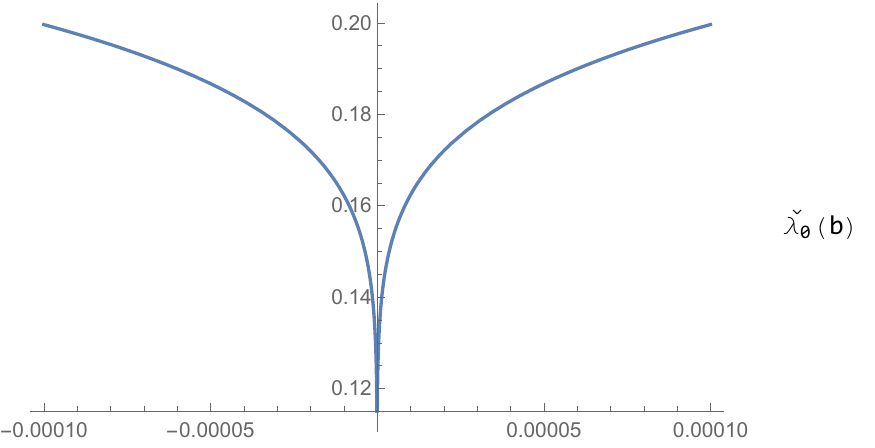}
	\end{center}
	\caption{The magnetic Steklov eigenvalues $\check \lambda_0(b)$.}
\end{figure}

\noindent
This is interesting to compare with the results of \cite{CG} where the authors consider the D-to-N operator for exterior problem with $-\Delta +p$ in the limit $p\rightarrow +0$.

\vspace{0.2cm}
\noindent
We recall that the eigenvalues of the Dirichlet to Neumann map  $\Lambda(b)$ on the interior of the unit disk  are given  by
\begin{equation}\label{explicitvprappel}
	\lambda_n (b) =  n - b +  2b \ \frac{M'(\half, n+1, b)}{M(\half, n+1,b)}\quad , \quad n \geq 0\,.
\end{equation}
For $n\leq 0$, thanks to symmetries,  we have 	$\lambda_n (b) :=	\lambda_{-n} (-b)$. In particular, when the magnetic field $b=0$, we recover the well-known result for the free Laplacian $-\Delta$ on the disk:
\begin{equation}
\lambda_n (0) = |n| \quad , \quad n \in \Z\,.
\end{equation}

\vspace{0.2cm}
\noindent
We now prove the following result:

\begin{thm}\label{weaklimi}
	 For any $b>0$, $\check \Lambda (b)-\Lambda(0) \in {\cal{B}}(L^2(S^{1}))$ and we have
	\begin{equation}
	|| \check \Lambda(b)-\Lambda (0)||_{ {\cal{B}}(L^2(S^{1})) } = \mathcal O(\frac{1}{|\log b|}) \quad , \quad b \to 0^+.
	\end{equation}
\end{thm}

\begin{proof}
Clearly, one has to prove that
\begin{equation}\label{goal}
|\ \check \lambda_n (b) -|n|\  | =  \mathcal O(\frac{1}{|\log b|}) \quad , \quad b \to 0^+,
\end{equation}
uniformly with respect to $n \in \Z$. For simplicity, we restrict ourselves to the case $n \geq 0$. 
First, for small values of $n$, the estimates (\ref{goal}) comes directly from (\ref{explicitvpext1}) and the asymptotic expansions  (\ref{asymptotUzero}), as it has be done in (\ref{lecasn=0}). For instance, one gets:
\begin{eqnarray}
	\check \lambda_1 (b) -1 &=& \mathcal O(b |\log b|) ,\\
	\check \lambda_n (b) -n &=& \mathcal O(b) \quad {\rm{for \ any\ fixed} \ n\geq 2.}
\end{eqnarray}
Now, to get uniform estimates, we use the well-known Laplace method. We obtain (see \cite{Te2}, Eq. (10.4.90))  the following asymptotic expansion which is uniform with respect to $z \in (0,1]$:
\begin{equation}\label{asymtUn}	
	z^n \ U(\half, n+1,z) = \sqrt{2} \ n^{n-\half} \left( 1-\frac{z}{n+1}\right)^{-\half} e^{z-n-1} \ \left(1+ \mathcal O(\frac{1}{n})\right) \quad , \quad n \to + \infty.
\end{equation}	
Moreover, following the proof given in  (\cite{Te2}), it is easy to see that the previous asymptotics can be derivated with respect to $z$.  Thus, taking the logarithmic derivative of (\ref{asymtUn}) with respect to $z$, we get, uniformly for $z \in (0,1] $,
\begin{equation}\label{derivlog}
\frac{n}{z} + \frac{U'(\half, n+1,z)}{U(\half, n+1,z) } = 1+ \mathcal O(\frac{1}{n})\,.
\end{equation}
Then, using (\ref{explicitvpext}), we get, uniformly for $b\in (0,1] $,  
\begin{equation}
	\check \lambda_n (b) = -n+b -2b \ \left( -\frac{n}{b} +  \mathcal O(1) \right) \mbox{ as } n \to +\infty\,.
\end{equation}
 In other words, we have $\check \lambda_n  - n = \mathcal O(b)$ uniformly for large $n$ and the proof is complete.
\end{proof}

\vspace{0.2cm}
\noindent
\begin{rem}
Actually, using the previous asymptotics of  $U(a,c,z)$ as $z \to 0^+$, we can get a more accurate asymptotic estimate for $\check \lambda_n (b)- |n|$ as in (\ref{lecasn=0}). For instance, we have
\begin{eqnarray}
\check \lambda_1 (b) -1  &=& b \log b  + \mathcal O(b), \\ 
\check \lambda_n (b) - n &=& - \frac{n}{n-1} \ b + \mathcal O(b^2) \ ,\ n \geq 2\,.
\end{eqnarray}
It is not clear for us that these asymptotics are uniform with respect to $n$.
\end{rem}

\subsection{Intersecting points for the case  outside  the disk.}

As  in \cite{HN} (see also \cite{HeLe}), the goal is to determine the intersection points between the curves of the magnetic Steklov eigenvalues $\check \lambda_n (b)$ and $\check\lambda_{n+1}(b)$. We restrict  our analysis  to the case of positive intersection points and $b$ is replaced by the variable $z$.  

\vspace{0.2cm}\noindent 
Although the scheme of the proof is the same as in the case of the disk (\cite{HN}) (see the reminder above), we can not avoid to redo the details of the computations which are different, as well as some remainder estimates.

\subsubsection{Characterization of the intersection points.}
\vspace{0.2cm}\noindent
Let $\check z_n$ be the positive intersection point between the curves $\check \lambda_n (b)$ and $\check \lambda_{n+1} (b)$. In other words, one has:  
 \begin{equation} \label{intersection}
   \check\lambda_n(\check z_n)=\check\lambda_{n+1} (\check z_{n})\,.
\end{equation}
Using \eqref{explicitvpext1} we obtain immediately:
\begin{equation}\label{eq:fn1ext}
 \frac{U(\frac 12,n+2,z)- zU(\frac 32,n+3,z)}{U(\frac 12,n+2,z)} = -z \ \frac{U(\frac 32,n+2,z)}{U(\frac 12,n+1,z)}\,.
 \end{equation}
First, let us study the numerator in the left hand side (LHS) of \eqref{eq:fn1ext}.  Using (\ref{ContiguousU}b)  with the parameters $a = 3/2$ and $c=n+2$, we get:
\begin{equation}
U(\frac 12,n+2,z)-z \ U(\frac 32,n+3,z)+ (n+\frac 12)\ U(\frac 32,  n+2,z)=0\,.
\end{equation}
Hence we have:
 \begin{equation}
 (LHS)= - (n+\frac 12) \ \frac{ U(\frac 32,n+2,z)}{U(\frac 12,n+2,z)}\,.
\end{equation}
Since $U(\frac 32,n+2,z) \not=0$ thanks to the integral representation  (\ref{eq:integralrepU}), we consequently obtain at the intersection point:
\begin{equation}\label{eq:intersaa}
(n+\frac 12)\ U(\frac 12,n+1,z)= z  \ U (\frac 12,n+2,z)\,.
\end{equation}
Now, using (\ref{ContiguousU}b) with $a=\frac 12$, $c=n+1$, we get:
\begin{equation}\label{eq:intersz}
 U(-\frac 12, n+1,z) +(n + \frac12)\, U (\frac 12,n+1,z) - z\ U(\frac 12,n+2,z)=0\,.
\end{equation}
Thus, we get $\check\lambda_n(z)=\check\lambda_{n+1} (z)$ if and only if $U(-\frac 12,n+1,z)=0$. 

\vspace{0.4cm}\noindent
Hence we have the following result (to compare with Proposition 4.1 in \cite{HN}):
 \begin{prop}\label{characteriztionext}
For any $n\geq 0$, there is a unique positive intersection point $\check z_n$  between the curves $\check \lambda_n (b)$ and $\check \lambda_{n+1}(b)$. Moreover, one has:
\begin{equation}
 U(-\frac 12,n+1,\check z_n)=0\,.
\end{equation}	
\end{prop} 
	
\begin{proof}  For $x \in \R_+^*$, we define $f(x) = U(-\half; n+1, x)$. So, we get immediately $f'(x)= \half \ U(\half, n+2,x)$ which is positive thanks again to the integral representation  (\ref{eq:integralrepU}). Thus, $f$ is strictly increasing on $(0,+\infty)$. Now, using (\ref{asymptotUzero}) and $\Gamma(-\half)= -2 \sqrt{\pi}$, as well  (see \cite{MOS1966}, p. 288) as the following asymptotic expansions  as $x \to 0^+\,$:
\begin{equation}
U\left(-\half,n+1,x\right)=- \frac{(n-1)!}{2\sqrt{\pi}}\ x^{-n} +\mathcal O\left(x^{1-n}\right) \quad , \quad n \geq 2\,,
\end{equation}
we see that $U(-\half; n+1, x) \to -\infty$ as $x \to 0^+$, (see Figure 4 in the case $n=4$). Moreover, using (\ref{asymptotiqueU}), we get $U(-\half; n+1, x) \to +\infty$ as $x \to +\infty$. This concludes the proof.
\end{proof}

\begin{figure}
	\begin{center}
		\includegraphics[width=0.48\textwidth]{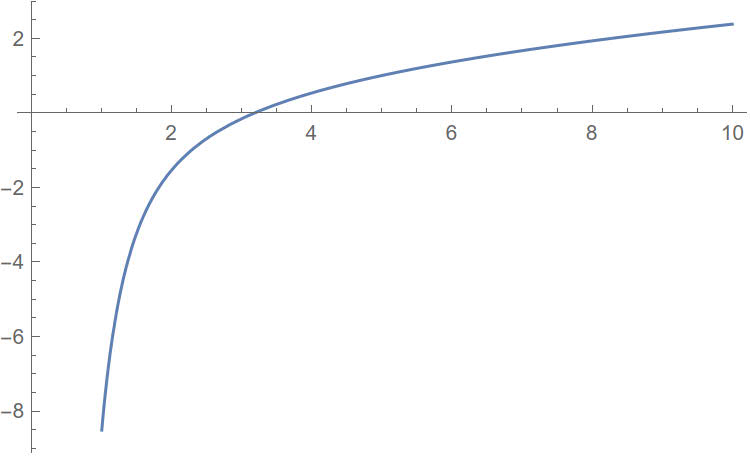}
	\end{center}
	\caption{The graph of $U(-\half, 5, x)$.}
\end{figure}

 \subsubsection{The (G) formula.}
 
 We want to compute for $z=\check z_n$
 \begin{equation}
 \check \lambda_n (z)= -n + z + z \frac{U(\frac 32,n+2,z)}{U(\frac 12,n+1,z)}\,.
 \end{equation}
 We use (\ref{ContiguousU}a)  with the parameters  $a=\frac 12$ and $c=n+2$ and  get
\begin{equation}\label{eq:intersy}
U(\frac 32,n+2,z)= 2U(\frac 12,n+2,z) - 2U(\frac 12,n+1,z)\,.
\end{equation}
Combining with \eqref{eq:intersz} and using Proposition \ref{characteriztionext}, we get, for $z=\check z_n$, 
\begin{equation}
 z \ U(\frac 32,n+2,z)= 2z \ U(\frac 12,n+2,z) - 2z \ U(\frac 12,n+1,z)= (2n+1-2z) \ U(\frac 12,n+1,z)\,.
\end{equation}
Coming back to the definition of  $\check\lambda_n$, we obtain (similarly to the (F)-formula in \cite{HN}):
 \begin{equation}\label{eq:magic}
(G)\quad  \check \lambda_n (\check z_n) = n+1 - \check z_n \,.
 \end{equation}
 
Let us also observe that, using $(G)$-formula and since D-to-N map is a positive operator,  one has, for any $ n \geq 0$,
\begin{equation}\label{eq:ineq}
	\check z_n <  n+1\,.
\end{equation}

\subsubsection{Computation of $\check \lambda'_n(z)$ and application to strong diamagnetism.}

Using (\ref{explicitvpext}), one has:
\begin{equation}
	\check \lambda_n(z) = -n + z - 2z \ \frac{U'(\half, n+1, z)}{U(\half, n+1, z)} \,.
\end{equation}
In what follows, to simplify the exposition, we set $U = U=U(\frac 12,n+1,z)$. By differentiation, we get immediately
\begin{equation}
	\check \lambda'_n(z)  = \frac{U^2 -2UU'-2zU''U +2zU'^2 }{U^2}\,.
\end{equation}
Since the confluent hypergeometric function  $U$ satisfies the differential equation (\ref{KummerODE}):
\begin{equation}
z U'' + (n+1-z) U' - \frac 12 U=0\,,
\end{equation}
we get 	
\begin{equation}\label{prime}
	\check \lambda'_n(z)=  2 \ \frac{\left((n-z)\ U +z\ U'\right)U'}{U^2}\,.
\end{equation}
Let us examine the numerator of the (RHS) of (\ref{prime}). Using (\ref{derivU}), we have:
\begin{equation}
	(n-z) \ U+z\  U' = 	(n-z)\ U(\frac 12,n+1,z) -\half \ U(\frac 32,n+2,z).
\end{equation}
Then, using \eqref{eq:intersy}, we get
\begin{eqnarray}
	(n-z)\  U +z\  U' &=& n\ U(\frac 12,n+1,z) - z\ U(\frac 12,n+2,z) \\
               	&=&  - ( U(-\half, n+1,z) +\half \ U(\half, n+1,z)),
\end{eqnarray}
where we have used \eqref{eq:intersz} in the last equality. As a conclusion, using again (\ref{ContiguousU}a)  with the parameters $a=-\frac 12$ and $c=n+1$, we have obtained:

\vspace{0.2cm}

\begin{prop}\label{characteriztionU}
For $ n \geq 0$, we have 
\begin{eqnarray}\label{eq:lambdaprimeU}
\check \lambda'_n(z) &=&  -2\ \frac{U'(\half, n+1, z)\ ( U(-\half, n+1, z) + \frac 12 \ U(\half, n+1, z) )}{(U(\half, n+1, z))^2}     \,.  \\ 
&=& -2\ \frac{U'(\half, n+1, z) \, U(-\frac 12,n,z) }{(U(\half, n+1, z))^2}     \,. \label{lambdaprime2U}
\end{eqnarray}
\end{prop}

\vspace{0.2cm}
\noindent
Now, using the integral representation formula (\ref{eq:integralrepU}), it is easy to see that  $$U'(\half, n+1,z) = -\half\  U(\frac32, n+2,z) <0$$ for positive real $z$. Moreover,  for $n \geq 1$, thanks to Proposition \ref{characteriztionext}, the map  $z \mapsto U (-\frac 12,n,z)$ is increasing with a unique zero at $\check z_{n-1}$. Thus, we obtain:

\begin{prop}
For $n \geq 1$, one has:
\begin{itemize}
\item $\check \lambda'_n (z_{n-1})	=0$. 
\item $\check \lambda'_n (z)>0$ on $(\check z_{n-1}, + \infty)$ and $\check \lambda'_n (z)<0$ on $(0, \check z_{n-1})$. 
\item  $\check z_{n-1}$ is the unique minimum of $\check \lambda_n(z)$.
\item  $\check z_{n-1} < \check z_{n}$.
\item  
 $\check \lambda_{n}(z)$ is increasing between $\check z_{n-1}$ and $\check z_{n}$. 
 \item On the interval  $[\check z_{n-1},\check z_n]$,  we have 
$
\check \lambda^{DN}(z)= \check \lambda_n(z)\,.
$
\end{itemize}
\end{prop}

\vspace{0.1cm}\noindent
As a final result, we obtain: 
 		
\begin{thm}
The map $z \mapsto  \check \lambda^{DN}(z)$ is increasing on $(0,+\infty)$.
\end{thm}
Hence we have strong diamagnetism for the exterior problem of the disk.\\

\noindent In addition, we have:

\begin{prop}
	For $n \geq 1$, one has $\check \lambda_n'' (\check z_{n-1}) > 0$. 
\end{prop}

\begin{proof}
 Using Proposition \ref{characteriztionU}, one has $U(-\half, n, \check z_{n-1})=0$. Thus, using \eqref{eq:lambdaprimeU} a straightforward calculation shows that: 
\begin{equation}
\check \lambda_n'' (z_{n-1}) = -2 \ \frac{U'(\half, n+1, \check z_{n-1}) \ U'(-\half, n, \check z_{n-1}) }{(U(\half, n+1, \check z_{n-1}))^2}.
\end{equation}
Now, we have: 
\begin{equation}
	U'(-\half, n, \check z_{n-1}) =  \frac{1}{2} \ U(\half, n+1, \check z_{n-1})\,.
\end{equation}
It follows that:
\begin{equation}
\check \lambda_n''  (\check z_{n-1}) = -    \frac{U'(\half, n+1, \check z_{n-1})}{U(\half, n+1, \check z_{n-1})}\,.
\end{equation}
So, using (\ref{explicitvpext}),  we get immediately:
\begin{eqnarray}
\check \lambda_n'' (\check z_{n-1}) &=& \frac{\check \lambda_n (\check z_{n-1}) + n - \check z_{n-1}}{2 \check z_{n-1}}  
\nonumber \\
	&=&  \frac{\check \lambda_{n-1} (\check z_{n-1}) + n - \check z_{n-1}}{2 \check z_{n-1}} \nonumber \\ &= &  \frac{n -\check z_{n-1}}{\check z_{n-1}} >0\,,
\end{eqnarray}
where we have used the characterization of the intersection point, \eqref{eq:ineq} 	and the (G)-formula.
\end{proof}


 \subsubsection{Asymptotics of $\check z_n$. }
 
 
The next goal  is to prove the following asymptotic expansion (to compare with the asymptotics of $z_n$ in \cite{HN}):

\begin{prop}\label{completeasymptU}
As $n\rightarrow +\infty$, $\check z_n$ admits the asymptotics 
\begin{equation}\label{eq:twotermsU} 
\check z_n \sim n - \alpha  \sqrt{n}  +   \frac{\alpha^2 +2}{3} + \sum_{j\geq 1}\check  \alpha_j n^{-\frac j2}\,.
\end{equation}
where the $\check \alpha_j$'s  are  real constants.
\end{prop}

\vspace{0.2cm}\noindent
First, we can see that the Proposition \ref{completeasymptU} can be written as follows. Using (\ref{ContiguousU}a)  with the parameters $a=\frac 12$ and $c=n+2$, as well as (\ref{derivU}) and \eqref{eq:intersaa}, we see that 
$\check z_n$ is the unique solution of the equation 
\begin{equation}\label{eq:as1U}
n+ \frac 12 -z =- z\ \frac{U'(\frac 12,n+1,z) }{U(\frac 12,n+1,z)}\,.
\end{equation}
We introduce as new variable:
\begin{equation} \label{eq:as2U}
\check \beta = \frac{- z+ n+\frac 12}{\sqrt{n}}
\end{equation}
and we get that \eqref{eq:as1U} is equivalent to:

\begin{equation}\label{eq:as3U}
\check\Psi_n(\check \beta):=\check \beta\,  (1 - \frac{1}{\sqrt{n} } \check \theta_n(\check \beta)) + (1 + \frac{1}{2n})\;  \check \theta_n(\check \beta)=0\,,
\end{equation}
where
\begin{equation}\label{eq:as4U}
\check \theta_n(\check\beta) = \sqrt{n} \;  \ \frac{U'(\frac 12,n+1,z) }{U(\frac 12,n+1,z)}\,.
\end{equation}
Hence \begin{equation}\label{eq:defbetan}
\check \beta_n:= \frac{-\check z_n +n+ \frac 12}{\sqrt{n}}
\end{equation} is the unique solution of \eqref{eq:as3U} 
\begin{equation}
\check\Psi_n(\check \beta_n)=0\,,
\end{equation}
and Proposition \ref{completeasymptU} will be a consequence of  the existence of a sequence $\hat \alpha_j $ such that, as $n \rightarrow +\infty$, 
\begin{equation}\label{eq:twotermsUz} 
\check \beta_n \sim  \alpha    - \frac{2\alpha^2 +1}{6}  n^{-1/2} + \sum_{j\geq 2} \hat \alpha_j n^{- j/2}\,.
\end{equation}

\vspace{0.2cm}\noindent
Now, let us analyse  $\check \theta_n$.
Coming back to the formulas for $U(\frac 12,n+1,z)$ and $U'(\frac 12,n+1,z)$ and after a change of variable $s =\sqrt{n}\, t$ in the defining integrals, we obtain:
\begin{equation}
\check \theta_n(\beta) = -\frac{\check \sigma_n(\beta)}{\check \tau_n(\beta)}\,,
\end{equation}
where
\begin{eqnarray}
\check \sigma_n(\beta) &=& \int_0^{+\infty} e^{\big(\beta - n^{1/2}  - \frac 12 n^{-1/2} \big)s}\,s^{1/2} (1+s\,n^{-1/2})^{n-\frac 12} \ ds\,,\\
\check \tau_n(\beta) &=& \int_0^{\infty} e^{\big(\beta - n^{1/2}  - \frac 12 n^{-1/2} \big)s}\,s^{-1/2} (1+s\,n^{-1/2})^{n-\frac 12} \ ds\,.
\end{eqnarray}
We now treat the asymptotics for $\check \sigma_n$ and $\check \tau_n$ separately but focus on $\check \sigma_n$ since the proof for $\check \tau_n$ is identical.
We write the previous formula in the form:
\begin{eqnarray}
\check \sigma_n(\beta) &=& \int_0^{+\infty} e^{\big(\beta - n^{1/2}  - \frac 12 n^{-1/2} \big)s + (n-\frac 12) \log (1+s n^{-1/2} )}\,s^{1/2} \ ds \,.
\end{eqnarray}
Assuming that the critical zone is with $\beta$ bounded and $s \leq C n^{\frac 14}$, 
we  write
\begin{equation}\label{Taylor}
\big(\beta - n^{1/2}  - \frac 12 n^{-1/2} \big)s + (n-\frac 12) \log (1+s n^{-1/2} ) =  
\beta s -\frac{s^2}{2} +(\frac{s^3}{3} -s) n^{-\half} + \mathcal O(n^{-1}).
\end{equation}
and the main term is formally 
\begin{equation}
\check \sigma_n (\beta) = \int_0^{+\infty} \exp (\beta s -\frac{s^2}{2}) s^{\frac 12} \,  ds + o(1) \mbox{ as } n\rightarrow +\infty.
\end{equation}
Similarly, we formally get 
\begin{equation}
\check \tau_n (\beta) = \int_0^{+\infty} \exp (\beta s -\frac{s^2}{2}) s^{-\frac 12} \,  ds + o(1) \mbox{ as } n\rightarrow +\infty.
\end{equation}
To justify the main term and have a better control of the remainder, we first decompose the integral  from $C n^{1/2}$ to $+\infty$ and then from  $0$ to $C n^{\frac 12}$. A rough estimate shows that
\begin{equation}
\check \sigma_n(\beta) = \int_0^{Cn^{1/2}} e^{\big(\beta -n^{1/2}  - \frac 12 n^{-1/2} \big)s} \,s^{1/2} (1+s\,n^{-1/2})^{n-\frac 12} ds + \mathcal O (n^{-\infty})\,.
\end{equation}
Here we have used that
\begin{equation}  
\big(\beta - n^{1/2}  -  \frac 12 n^{-1/2} \big)s +(n-\frac 12)  \log (1+ s\,n^{-1/2})\leq  \beta s  + n (- sn^{-1/2} +  \log (1+ s\,n^{-1/2}))\,,
\end{equation}
and  the fact that there exist a constant $C>0$ such that
$$
x- \log (1+x) \leq -1\,,\, \forall x \geq C\,.
$$
For the remaining zone, we decompose the integral in two subzones from $0$ to $n^{\rho}$ and  from $n^{\rho}$ to $C n^{1/2}$ with $\rho < \frac 12$.  The new difficulty (in comparison with \cite{HN}) is to control the integral on the second subzone.

\vspace{0.2cm}
\noindent
We come back to the sign of $-x + \log (1+x)$ and for $\epsilon >0$ small enough we want to determine when the inequality $-x + \log (1+x) < - \epsilon$ holds. One can show that there exists $$x(\epsilon)\sim \sqrt {2\epsilon}$$ such that
\begin{equation}
-x + \log (1+x) < - \epsilon \,,\, \forall x > x(\epsilon)\,.
\end{equation}
We now take $\epsilon = n^{-\frac 12}$ (and $n$ large enough),  and this shows that there exists $\check  C$ such that if $s \geq \check C  n^{\frac 14}$ we have
\begin{equation}
 n (- sn^{-1/2} +  \log (1+ s\,n^{-1/2}) )\leq - n \epsilon =- n^{\frac 12}\,.
\end{equation}
It is then easy  to control the second subzone with $\rho =\frac 14$. 

\vspace{0.2cm}
\noindent
At last, the treatment of the first subzone is identical to the case of the interior of the disk and is obtained by a Taylor's expansion of $\log (1+x)$ with $x=s n^{-1/2}$.
 More precisely, we now consider
\begin{equation}
\check \sigma_n^{(1)}(\beta): = \int_0^{n^{\rho}} e^{\big(\beta - n^{1/2}  - \frac 12 n^{-1/2} \big)s}\,s^{1/2} (1+s\,n^{-1/2})^{n-\frac 12} ds\,.
\end{equation}
Using a Taylor expansion with remainder  of $\log (1+s\,n^{-1/2})$ to a sufficient high order, we get  an infinite sequence of polynomials $\check P_j$ ($j\geq 1$) such that for any $N$, there exists $p(N)$ such that 
\begin{equation}
\check \sigma_n^{(1)}(\beta) = \int_0^{n^{\rho}} e^{\beta s -\frac{s^2}{2}} s^{1/2} \Big (1 + \sum_{j=1}^{p(N)} \check P_j(s) n^{-j/2}\Big) ds + \mathcal O (n^{-N}) \,.
\end{equation}
In the last step, we see that, modulo an exponentially small error, we can integrate over $(0,+\infty)$ in order to get
\begin{equation}
\check \sigma_n(\beta)=\int_0^{+\infty} e^{\beta s -\frac{s^2}{2}} s^{1/2} \Big (1 + \sum_{j=1}^{p(N)} \check P_j(s) n^{-j/2}\Big) ds + \mathcal O (n^{-N})\,.
\end{equation}

\vspace{0.1cm}\noindent
Hence we get by integration the following  lemma: \\
\begin{lemma}
 For any $N$, there exist $p(N)$, $\check P_1(s)$    and $C^\infty$-functions $\check {\check {\sigma}}_{j}$  such that
 \begin{equation}
 \check \sigma_n(\beta)=\int_0^{+\infty} e^{\beta s -\frac{s^2}{2}} s^{1/2} ds + \Big(\int_0^{+\infty} e^{\beta s -\frac{s^2}{2}} s^{1/2} \check P_1(s)  ds\Big)\, n^{-1/2} + \sum_{j=2}^{p(N)}\check{ \check \sigma}_{j} (\beta) n^{-j/2} +  \mathcal O (n^{-N})\,.
 \end{equation}
Similarly,  for any $N$, there exist $p(N)$, $\check Q_1(s)$    and $C^\infty$-functions $\check{\check \tau}_{j}$  such that
\begin{equation}
\check  \tau_n(\beta)=\int_0^{+\infty} e^{\beta s -\frac{s^2}{2}} s^{-1/2} ds + \Big(\int_0^{+\infty} e^{\beta s -\frac{s^2}{2}} s^{-1/2} \check Q_1(s)  ds\Big) \, n^{-1/2}  + \sum_{j=2}^{p(N)} \check{\check \tau}_{j} (\beta) n^{-j/2} +  \mathcal O (n^{-N})\,.
\end{equation}
\end{lemma}

\vspace{0.1cm}\noindent
Looking at the first term in the Taylor expansion in (\ref{Taylor}), we see that
\begin{equation}
\check P_1(s)= \frac{s^3}{3}-s= \check Q_1(s)\,.
\end{equation}
Notice that in comparison with \cite{HN}, we have $\check P_1 =-P_1$.

\vspace{0.4cm}
\noindent
We can deduce a first localization of $\check\beta_n$.

\vspace{0.1cm}

\begin{lemma}
	For any $\eta >0$, there exists $n_0$ such that for $n\geq n_0$, 
	\begin{equation}
		\check\beta_n \in [\alpha -\eta,\alpha + \eta]\,.
	\end{equation}
\end{lemma}

\begin{proof}
	For fixed $\beta$, we have
	\begin{equation}
		\lim_{n\rightarrow +\infty}  \check\sigma_n(\beta)= \int_0^{+\infty} e^{\beta s -\frac{s^2}{2}} s^{1/2} ds\,,
	\end{equation}
	and
	\begin{equation}
		\lim_{n\rightarrow +\infty}  \check\tau_n(\beta)= \int_0^{+\infty} e^{\beta s -\frac{s^2}{2}} s^{-1/2} ds\,,
	\end{equation}
	This implies, 
	\begin{equation}
		\check\Phi(\beta):=\lim_{n\rightarrow +\infty}  \check\theta_n(\beta) = - \frac{\int_0^{+\infty} e^{\beta s -\frac{s^2}{2}} s^{1/2} ds}{\int_0^{+\infty} e^{\beta s -\frac{s^2}{2}} s^{-1/2} ds}\,,
	\end{equation}
	or equivalently using (\ref{integralrep}),
	\begin{equation}
		\check\Phi(\beta) = - \half \frac{D_{-3/2} (-\beta)}{D_{-1/2}(-\beta)} \,,
	\end{equation}
	(we remark that, for any  $\nu<0$, $D_{\nu} (z)$ has no real zeros since the integrand in (\ref{integralrep}) is always positive).
	
	\vspace{0.1cm}\noindent
	Now, using (\ref{recurrenceDnu1}) with $\nu =-\half$ and $z=-\beta$, we immediately get:
	\begin{equation}\label{newPhi}
		\check\Phi (\beta) = -\beta - \frac{D_\half ( -\beta)}{D_{-\half} (-\beta)}.
	\end{equation}
  Considering now $\check\Psi_n$ given in (\ref{eq:as3U}), one gets:
	\begin{equation}
		\lim_{n\rightarrow +\infty} \check\Psi_n(\beta) = \beta + \check \Phi(\beta) =  - \frac{D_\half ( -\beta)}{D_{-\half} (-\beta)}\,.
	\end{equation}
	Since  $D_\half ( -\alpha)=0$, one has $D_\half '( -\alpha)\neq 0$ since $D_\nu(z)$ satisfies the second order ordinary differential equation (\ref{ODEDnu0}). It is then clear that for $\eta >0$ small enough, and $n$ large enough we have
	\begin{equation}
		\check\Psi_n(\alpha -\eta) \check\Psi_n(\alpha +\eta) <0\,.
	\end{equation}
	Hence $\check\Psi_n$ should have a zero in this interval, which is necessarily $\check z_n$ by uniqueness.
	This achieves the proof of the lemma.
\end{proof}

\vspace{0.1cm}\noindent
In other words, we have shown that
\begin{equation}\label{twoterm}
	\lim_{n\rightarrow +\infty} \check\beta_n =\alpha\,,
\end{equation}
and as a consequence we obtain  a two-terms asymptotics for $\check z_n$.

\subsubsection{A complete asymptotic expansion.}

We refer the reader to \cite{HN} for the end of the proof since the arguments are strictly identical. Indeed, we see that $\check \Phi(\alpha) =-\Phi(\alpha)$ where $\Phi(\alpha)$ is the analogous function appearing in the interior case. 
In particular, we get: 

\begin{equation}
	\check \beta_n \sim \alpha - \frac{2\alpha^2 +1}{6} \ n^{-\half} + \sum_{j\geq 1} \check \alpha_j \ n^{-\frac {j+1}{2}}\,,
\end{equation}
which concludes the proof recalling from \eqref{eq:defbetan} that $\check z_n = n+\half - \sqrt{n} \ \check \beta_n$. \hfill $\Box$

\subsubsection{Applications.}
\vspace{0.2cm}\noindent
Now, mimicking the proofs of \cite{HN}, we can first  show:
\begin{prop}
	\begin{equation}
		\lim_{z\rightarrow +\infty} z^{-1/2} \, \check \lambda^{DN} (z) = \alpha\,.
	\end{equation}
\end{prop}

\noindent
Then, we deduce successively, exactly as in \cite{HN}, the following asymptotic expansions.

\begin{prop}\label{ecart}
	We have
	\begin{eqnarray}
		\check z_n- \check z_{n-1} &=& 1 - \frac{\alpha}{2} n^{-1/2} + \mathcal O (n^{-1})\,, \\ 
		\check\lambda_n( \check z_n)  &=&\alpha \, n^{1/2} + \frac{1- \alpha^2}{3}  + \mathcal O (n^{-1/2}) \,,\\
		\check \lambda^{DN} (z)&=& \alpha z^{1/2} + \frac{\alpha^2 +2}{6} + \mathcal O (z^{-1/2})\,.
	\end{eqnarray}
\end{prop}

\section{Additional  flux effects.}\label{ABsection}

\subsection{Introduction}
First, let us  consider for $\nu \in (-\half, \half]$  the (A-B) potential  in the exterior of the disk $\Omega = R^2\setminus D(0,1)$ introduced in \eqref{defpotAB}.
In this section, we would like to analyze, as a function of  the magnetic flux $\nu$ and the magnetic field $b$, the ground state of the D-to-N operator  associated with the magnetic potentiel 
\begin{equation}
	 A_{b,\nu}(x,y) := A^b(x,y) + 	A_\nu (x,y) = \left( b + \frac{\nu}{r^2}\right) (-y,x)\,.
\end{equation}
Clearly, the associated  magnetic field is constant of strengh $2b$ in $\Omega$.

\vspace{0.2cm}
\noindent
In order to define the D-to-N map denoted $\check \Lambda(b,\nu)$, we solve as previously:
\begin{equation}  \label{DirichletAB}
	\left\{
	\begin{array}{rll}
		H_{A_{b,\nu}} \  v &=&0  \  \ \rm{in}  \ \ \Omega,\\
		v_{  \vert \partial \Omega} & =& \Psi \in H^{1/2}(\partial\Omega) .
	\end{array}\right.
\end{equation}
Working with the polar coordinates $(r, \theta)$ and using the Fourier decomposition   
\begin{equation}
	v(r, \theta) = \sum_{n \in \Z} v_n (r) e^{in \theta}\ \ , \ \  \Psi(\theta) = \sum_{n \in \Z} \Psi_n e^{in \theta}, 
\end{equation}
a staightforward calculation shows that $v_n (r)$ solves the following ODE:
\begin{equation}\label{polarequationsAB}
	\left\{
	\begin{array}{ll}
		- v_n'' (r) - \frac{v_n'(r)}{r} + (br-\frac{n-\nu}{r} )^2 v_n (r)= 0   & \mbox{for} \  r >1 , \\
		v_n (1) = \Psi_n .
	\end{array}
	\right.
\end{equation}

\vspace{0.2cm}
\noindent
If the magnetic field $b$ is positive, the unique bounded solution at infinity is given by: 
\begin{equation}
	v_n (r)= c_{n,\nu} \  e^{- b r^2/2} r^{n-\nu} \ U (\frac 12,n-\nu+1, br^2)\,,
\end{equation}
where $c_{n,\nu}$ is a suitable constant.\\
 Then, we easily get for the eigenvalues of the D-to-N map $\check \Lambda(b,\nu)$:
\begin{equation}\label{explicitvpextAB}
	\check \lambda_n (b, \nu) = \nu -n+b -2b \ \frac{  U' (\frac 12,n - \nu +1, b) }{ U (\frac 12,n-\nu+1, b) }\,,
\end{equation}
or equivalently
\begin{equation}\label{explicitvpext1AB}
	\check \lambda_n (b, \nu) = \nu -n+b +b \ \frac{  U (\frac 32,n - \nu +2, b) }{ U (\frac 12,n-\nu+1, b) }\,.
\end{equation}
When the magnetic field $b$ is negative, thanks to symmetries in  (\ref{polarequationsAB}), we get immediately: 
\begin{equation}\label{symetriesAB}
	\check \lambda_n (b, \nu):=\check \lambda_{-n} (-b, -\nu)\,.
\end{equation}

\begin{figure}
	\begin{center}
		\includegraphics[width=0.48\textwidth]{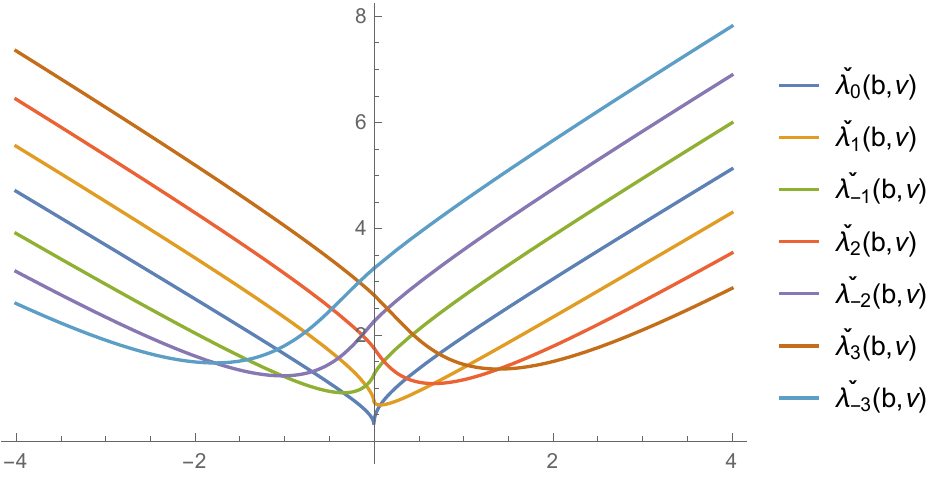}
		\includegraphics[width=0.48\textwidth]{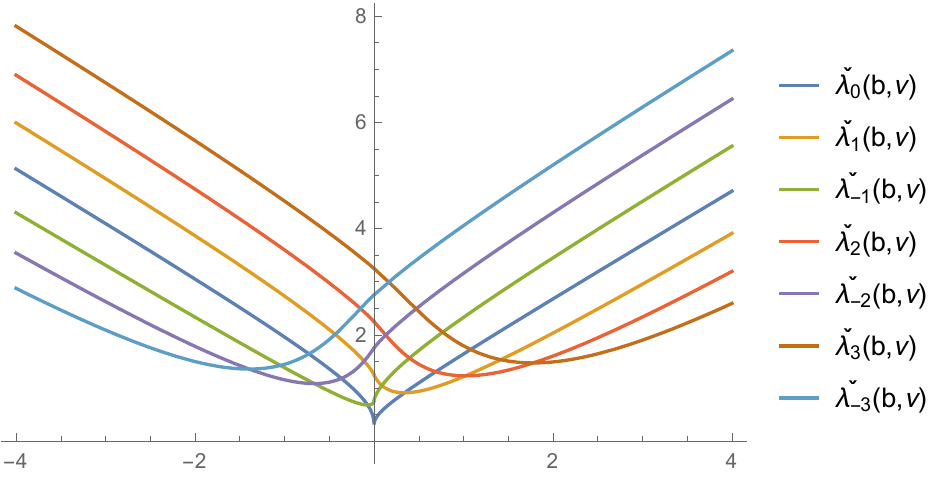} 
	\end{center}
	\caption{The magnetic Steklov eigenvalues $\check \lambda_n(b, \nu)$ for $\nu = \frac14$ (left) and for $\nu=-\frac14$  (right).}
\end{figure}

\subsection{Weak magnetic field limit}
As in Section \ref{s4}, we are interested in the weak-field limit, i.e when the magnetic field  $b \to 0^+$. We shall see that the results are sightly different. The only difference compared to the case studied previously in Section \ref{s4} 
 is that $n-\nu$ is no longer an integer. Therefore, the logarithmic terms appearing in the weak-field asymptotic expansions in Section \ref{s4}  disappear. Indeed, we recall (\cite{MOS1966}, p. 288 or \cite{SoSo}) that the following hypergeometric functions  satisfy the asymptotics as $z \to 0^+$: 
\begin{subequations}\label{asymptotUzeroAB}
\begin{eqnarray}
	U\left(a,c,z\right)&=&\frac{\Gamma(1-c)}{\Gamma\left(a+1-c\right)} + \mathcal O\left(z \right) \quad ,\quad c<0, \\
	U\left(a,c,z\right)&=&\frac{\Gamma(1-c)}{\Gamma\left(a+1-c\right)} + \mathcal O\left(z^{1-c}\right) \quad ,\quad 0<c<1, \\
	U\left(a,c,z\right)  &=& \frac{\Gamma(c-1)}{\Gamma\left(a\right)} z^{1-c}+ \mathcal O\left(1\right) \quad ,\quad 1<c<2, \\    
	U\left(a,c,z\right) &=& \frac{\Gamma(c-1)}{\Gamma\left(a\right)} z^{1-c} + \mathcal O\left(z^{2-c}\right) \quad ,\quad c>2.
\end{eqnarray}
\end{subequations}
\begin{subequations} \label{lecasAB}
Then, for  $\nu \in (-\half, \half] \backslash \{ 0\}$, and thanks to the symmetries (\ref{symetriesAB}), we get the following asymptotics for  the eigenvalues $\check \lambda_n(b, \nu)$ as $b \to 0^+\,$:
\begin{eqnarray}
	\check \lambda_{-1}(b,\nu) &=&  \nu + 1 + \mathcal O(b^{1-|\nu|}) , \\
	\check \lambda_0(b,\nu) &=& |\nu| + \mathcal O(b^{|\nu|}) , \\
	\check \lambda_1(b,\nu) &=& 1-\nu + \mathcal O(b^{1-|\nu|}) , \\
	\check \lambda_n(b,\nu) &=& |n-\nu| + \mathcal O(b) \quad , \quad |n|\geq2.
\end{eqnarray}
\end{subequations}

\vspace{0.2cm}
\noindent
 Figure~5 represents the graphs of the Steklov eigenvalues $\check \lambda_n(b,\nu)$ for $\nu=\frac14$ and $\nu = -\frac14$.
At last, we see that  $\lambda_0(0,\half) = \lambda_1(0,\half)$ is a double eigenvalue, (see Figure 6 below).

\begin{figure}
	\begin{center}
		\includegraphics[width=0.48\textwidth]{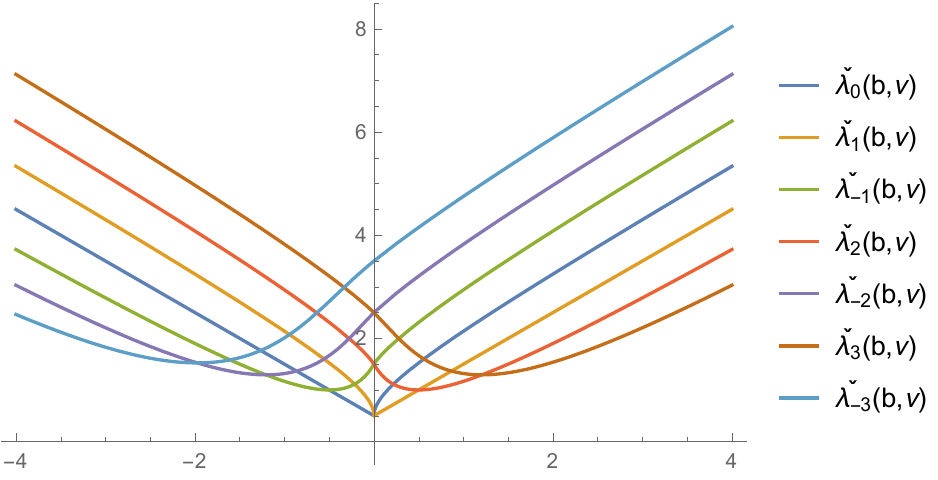}
		
	\end{center}
	\caption{The magnetic Steklov eigenvalues $\check \lambda_n(b, \nu)$ for $\nu=\half$.}
\end{figure}

\newpage
\vspace{0.5cm}
\noindent
By mimicking the proof of Theorem \ref{weaklimi}, we get:

\vspace{0.2cm}

\begin{thm}\label{weaklimitAB}
	For any $b>0$ and any $\nu \in (-\half, \half] \backslash \{0\}$,  $\check \Lambda(b,\nu)-\widehat \Lambda(\nu) \in {\cal{B}}(L^2(S^{1}))$ and we have
	\begin{equation}
		|| \check \Lambda(b,\nu)-\widehat \Lambda(\nu)||_{ {\cal{B}}(L^2(S^{1})) } = \mathcal O(b^{|\nu|}) \mbox{ as } b \to 0^+\,.
	\end{equation}
\end{thm}

\subsection{Diamagnetism and strong magnetic field limit}
\vspace{0.2cm}
\noindent
Finally, we can prove the strong diamagnetism and we get the same asymptotic expansions of the ground state $\check \lambda^{DN}(b,\nu)$ of the D-to-N map  $\check \Lambda(b,\nu)$  exactly as before, (i.e when $\nu=0$). Indeed, in the previous sections, we do not use that $n$ was an integer. So writing our previous results with $n-\nu$ instead of $n$, all our propositions and theorems are  the same, except  Proposition \ref{completeasymptU} for the asymptotics of the intersection point denoted $\check z_n (\nu)$. We can easily prove:

\begin{prop}\label{completeasymptUflux}
	For any fixed $\nu \in (-\half, \half]$, $\check z_n (\nu)$ admits the asymptotics as $n\rightarrow +\infty$,
	\begin{equation}\label{eq:twoterms} 
		\check z_n (\nu) \sim (\n) - \alpha  \sqrt{\n}  +    \frac{\alpha^2 +2}{3} + \sum_{j\geq 1}\check  \alpha_j  \ (\n)^{-\frac j2}\,,
	\end{equation}
	where the $\check \alpha_j$ are independent of $\nu$.
\end{prop}
Of course, this implies 
\begin{equation}\label{eq:twotermsbis} 
		\check z_n (\nu) \sim n  - \alpha  \sqrt{n}  +    \frac{\alpha^2 +2}{3} -\nu + \sum_{j\geq 1}\check  \alpha_j (\nu)   n^{-\frac j2}\,,
	\end{equation}
but this will not be useful for the next computation.
\vspace{0.2cm}
\noindent

\vspace{0.2cm}\noindent
Now, thanks to the (G) formula (with $n$ replaced by $n-\nu$), we first get:
\begin{equation}\label{lambdan}
	\check \lambda_n(\check z_n (\nu), \nu)  =\alpha (\n)^{1/2} + \frac{1-\alpha^2}{3}  
	-   \frac{\check \alpha_1 }{\sqrt{\n}} + \mathcal O ((\n)^{-1}).
\end{equation}
Using Proposition \ref{completeasymptUflux}, a straightforward calculation gives:
\begin{equation}\label{difference}
 \check z_n (\nu)-\check z_{n-1} (\nu)= 1 - \frac{\alpha}{2\sqrt{\n}}  +  \mathcal O ((\n)^{-\frac32})\,. 
\end{equation}
\vspace{0.1cm}\noindent
Using again Proposition \ref{completeasymptUflux}, we get:
\begin{equation}\label{racine}
		\sqrt{\n} =   \sqrt{\check z_n(\nu)} + \frac{\alpha}{2}  
		- \frac{\alpha^2 +8}{24 \sqrt{\check z_n(\nu) }}
		+ \mathcal O ( \frac{1}{\check z_n(\nu)}).
\end{equation}
Then, plugging (\ref{racine}) into (\ref{lambdan}), we get:
\begin{equation}\label{asymptflux}
		\check \lambda_n(\check z_n (\nu), \nu)  = \alpha \sqrt{ \check z_n(\nu)}+ \frac{\alpha^2 +2}{6} -\left(\check \alpha_1  +  \frac{\alpha (\alpha^2+8)}{24} \right) \frac{1}{\sqrt{\check z_n (\nu) }} + \mathcal O (\frac{1}{\check z_n (\nu)}).
\end{equation}

\vspace{0.2cm}\noindent
Finally, like in the proof of the interior case \cite{HN} , we get
\begin{equation}\label{asymptwoterms}
\check \lambda(z,\nu) = \alpha z^{\frac 12} + \frac{\alpha^2+ 2}{6} + \mathcal O (z^{-1/2})\,.
\end{equation}

\section{Prospective and conjectures for general domains}

The  proof given in the Section 5.3, (like an alternative proof in the spirit of \cite{HKN}), does not permit to control the dependence in $\nu$ of the remainder in (\ref{asymptwoterms}). 
Our guess is that the third term of these asymptotics is oscillating as it is the case (see \cite{FH2,FPS0,HeLe}) for the Neumann eigenvalue asymptotics for the disk where we have:

\begin{thm}
With, for $m\in {\mathbb Z}$, $b>0$, $
\delta(m,b) = m - b - \sqrt{2\Theta_0 b}\,,
$
there exist (computable) constants $C_0, \ C_1,\ \delta_0 \in {\mathbb R}$ such that if
\begin{align}
\Delta_b = \inf_{m \in {\mathbb Z}} | \delta(m,b) - \delta_0|\;,
\end{align}
then, as  $b \to + \infty$,  the first magnetic Neumann eigenvalue satisfies 
\begin{align}\label{eq:asympFoHe}
\lambda^{Ne}(b) = 2\Theta_0 b -  C_1 \sqrt{2b} +
3 |C_1| \sqrt{\Theta_0} \big( \Delta_b^2 + C_0\big) + {\mathcal O}(b^{-\frac{1}{2}})\;.
\end{align}
Here, $\Theta_0$ is the so-called De Gennes constant, which is equal  (see \cite{Bo}) approximately  to 
\begin{equation}\label{eq:2.23}
	\Theta_0\approx 0.5901061249 ....
\end{equation}
\end{thm}
Notice that a similar statement is proven for the exterior of the disk in \cite{FPS0} (Theorems 1.10 and 4.1) with,  in  \eqref{eq:asympFoHe},  $\delta_0, C_0,C_1$ and $\Delta_b$ replaced by  some $\hat \delta_0, \hat C_0,\hat C_1,  \hat \Delta_{b,\nu}$, with $$\hat C_1=-C_1\,,\,
 \hat \Delta_{b,\nu}=  \inf_{m \in {\mathbb Z}} | \hat \delta(m-\nu,b) - \hat \delta_0|\,, \mbox{ and } 
\hat{\delta}(m,b) = m-b+\sqrt{2\Theta_0 b}\,.$$

 In \cite{HKN}, we prove with A. Kachmar: 

\begin{thm}
\label{thm:main-2D}
 	Let $\Omega$ be a regular domain in $\mathbb R^2$ and $A$ be a magnetic potential with constant magnetic field with norm $1$.  Then the ground state energy of the D-to-N map  $\Lambda^{DN}_{bA}$ satisfies as $b \rightarrow +\infty$
 \begin{equation*}
\lambda^{\rm DN}(b A,\Omega) =  \hat \alpha b^{\frac 12}  -   \frac{\hat \alpha^2 +1}{3}  \max_{x\in \partial \Omega} \kappa_x \, + o(1)\,,
 \end{equation*}
 where $\kappa_x$ denotes the curvature at $x$ and $\hat \alpha =\alpha/\sqrt{2}$.
\end{thm}

The extension to the exterior problem is rather immediate, by a  small variation of the proof given in \cite{HKN}, 

  \begin{thm}\label{prop:lb-2D-B=1}
Let $\Omega$ be a regular domain in $\mathbb R^2$ with compact boundary $\partial \Omega$ and $A$ be a vector potential with a positive magnetic field $B={\rm curl} A$  with uniform lower bound.
    Suppose that $B$  is $C^1$  on $\overline{\Omega}$ and that $B=1$  on a neighborhood of $\partial\Omega$. Then, the ground state energy of the D-to-N map $\Lambda_{b A}$ satisfies
 \begin{equation*}
 \lambda^{\rm DN}(b A,\Omega) =  \hat\alpha\, b^{\frac{1}{2}} -  \frac{\hat \alpha^2+1}{3}\, \max_{x\in\partial\Omega} \kappa_x +\mathcal O(b^{-1/6})\quad , \quad b\to+\infty\,,
 \end{equation*}
\end{thm}
Notice that the two first terms in the expansion depend only on the magnetic field and not on the generating magnetic potential. If $\tilde A$ and $A$ are two vector potentials with same magnetic field and satisfying the Coulomb gauge condition, it would be
 interesting to see, in the constant curvature case, how the remainder depends on the circulation of the tangential component of $A-\tilde A$ along $\partial \Omega$.

Finally the weak-field limit for general unbounded domains   seems completely  open (see nevertheless \cite{KLPS}).

\vspace{0.5cm}

\noindent \footnotesize{
	
	\noindent Laboratoire de Math\'ematiques Jean Leray, UMR CNRS 6629. Nantes Universit\'e  F-44000 Nantes  \\
	\emph{Email adress}: Bernard.Helffer@univ-nantes.fr \\

	\noindent Laboratoire de Math\'ematiques Jean Leray, UMR CNRS 6629. Nantes Universit\'e  F-44000 Nantes \\
	\emph{Email adress}: francois.nicoleau@univ-nantes.fr \\


\begin{thebibliography}{99}

 \bibitem{BW24} M. Baur and T. Weidl. 
 \newblock Eigenvalues of the magnetic Dirichlet-Laplacian with constant magnetic field on discs in the strong field limit.
 \newblock arXiv:2402.01474 (2024).  Anal. Math. Phys. 15 (2025), No. 1, Paper No. 9, 30 p.


 


 
 
\bibitem{Bo} V. Bonnaillie-No\"el.
\newblock Harmonic oscillators with Neumann condition of the half-line. 
\newblock Commun. Pure Appl. Anal. 11 (2012), no. 6, 2221--2237.


\bibitem{CGHP} T. Chakradhar,  K. Gittins, G. Habib, and N. Peyerimhoff.
\newblock A note on the magnetic Steklov operator on functions.
\newblock arXiv:2410.07462v1-v3 (2024-2025).

\bibitem{CPS} B. Colbois, L. Provenzano, and A. Savo.
\newblock Isoperimetric inequalities for the magnetic Neumann and Steklov problems with Aharonov-Bohm magnetic potential.
\newblock Journal of Geometric Analysis 32 (11): Paper 285, 38, (2022).

 \bibitem{CLPS1} B. Colbois, C. L\'ena, L. Provenzano, and A. Savo.
 \newblock Geometric bounds for the magnetic Neumann eigenvalues in the plane.
 \newblock  Journal de Math\'ematiques Pures et Appliqu\'ees 179, 454-497, (2023).
  
  
  
    
  
  


  
 
\bibitem{DLMF}
{\it NIST Digital Library of Mathematical Functions}.
\newblock https://dlmf.nist.gov/, Release 1.2.0 of 2024-03-15.
\newblock F.~W.~J. Olver, A.~B. {Olde Daalhuis}, D.~W. Lozier, B.~I. Schneider,
  R.~F. Boisvert, C.~W. Clark, B.~R. Miller, B.~V. Saunders, H.~S. Cohl, and
  M.~A. McClain, eds. 



\bibitem{EO}   A. F. M. ter Elst  and   El Maati Ouhabaz.
\newblock  The diamagnetic inequality for the Dirichlet-to-Neumann operator.
\newblock Bull. Lond. Math. Soc. 54 (2022), no. 5, 1978--1997.

  



 
 \bibitem{FH2} S. Fournais and B. Helffer.
 \newblock Strong diamagnetism for general domains and applications.
 \newblock Annales de l'Institut Fourier, Tome 57 (2007) no. 7, 2389-2400. 
 
 
 
 
 


 
 

\bibitem{FPS0} S. Fournais and M. Persson Sundqvist.
\newblock Lack of diamagnetism and the Little Park effect.
\newblock ArXiv 1405.4690v1 (2014). 
Comm. Math. Phys. 337 (2015), no. 1, 191--224.
 



\bibitem{GKPS} M. Goffeng, A.Kachmar, and M. Persson
Sundqvist.
\newblock  Clusters of eigenvalues for the magnetic Laplacian with Robin
condition. 
\newblock J. Math. Phys. 57 (2016), no. 6, 063510, 19 pp.


\bibitem{GP} A. Girouard and I. Polterovich.
\newblock Spectral geometry of the Steklov problem (Survey).
\newblock In book (Henrot editor).

\bibitem{CG} D. Grebenkov and A. Chaigneau.
\newblock The Steklov problem for exterior domains: asymptotic behavior and applications. 
\newblock ArXiv 2024.



\bibitem{HKN} B. Helffer, A. Kachmar, and F. Nicoleau.
\newblock  Asymptotics for the magnetic Dirichlet-to-Neumann eigenvalues in general domains.
\newblock ArXiv 2025.


\bibitem{HeLe} B. Helffer and C. L\'ena.
\newblock Eigenvalues of the Neumann magnetic Laplacian in the unit disk.
\newblock   arXiv:2411.11721, (2024).

\bibitem{HeMo} B. Helffer and A. Mohamed.
\newblock  Caract\'erisation du spectre essentiel de l'op\'erateur de Schr\"odinger avec un champ magn\'etique
\newblock Ann. Inst. Fourier, 38 (2) (1988), 95--112


\bibitem{HN} B. Helffer and F. Nicoleau.
\newblock  On the magnetic Dirichlet to Neumann operator on the disk -- strong diamagnetism and strong magnetic field limit-- 
\newblock ArXiv:2411.15522  v3,  (2024). 


\bibitem{Hor1} L. H\"ormander.\newblock  \emph{The Analysis of Linear Partial Differential Operators, III}. \newblock Grundlehren der matematischen Wissenschaften $256$, Springer-Verlag, Berlin Heidelberg, (1985).

\bibitem{Hor2} L. H\"ormander. \newblock Uniqueness theorems for second order elliptic differential equations. \newblock  Communications in Partial Differential Equations $8$ (1), (1983), 21-64.



 
 \bibitem{KLPS} A. Kachmar, V. Lotoreichik, and M. Persson-Sundqvist.
 \newblock On the Laplace operator with a weak magnetic field in exterior domains.
 \newblock ArXiv:2405.18154v1 (2024).  Anal. Math. Phys. 15 (2025), No. 1, Paper No. 5.

 \bibitem{KMi} A. Kachmar and G. Miranda. The magnetic Laplacian on the disc for strong magnetic fields.
 \newblock Arxiv:2407.11241v1 (2024).  J. Math. Anal. Appl. (2025), no. 2, Paper No. 129261. 
 
\bibitem{KPS} A. Kachmar and M. Persson. 
\newblock On the essential spectrum of magnetic
 Schr\"odinger operators in exterior domains.
 \newblock Arab J Math Sci 19(2) (2013), 217--222.

 \bibitem{Ki} Y. Kian.  \newblock Determination of non-compactly supported electromagnetic potentials in unbounded closed waveguide. \newblock  Revista Matem\`atica Iberoamericana,  36 (2020),   671-710.
 
 \bibitem {KrUh} K. Krupchyk, M. Lassas, and G. Uhlmann. 
 \newblock  Inverse problems with partial data for a magnetic Schr\"odinger operator in an Infinite slab or bounded domain.  
 \newblock Comm. Math. Phys.,   327 (2014),   993-1009.
 
 

\bibitem{MOS1966} W. Magnus, F. Oberhettinger, and R.P. Soni.
\newblock Formulas and theorems for the special functions of mathematical physics,  3rd enlarged ed, 
\newblock Grundlehren der Mathematischen Wissenschaften,
\newblock Volume 52,
\newblock Springer,  (1965).

\bibitem{Leb} N.N. Lebedev. 
\newblock Special functions and their applications
\newblock Prentice-Hall
\newblock  Englewood Cliffs (1965).

\bibitem{Li} X. Li. 
\newblock  Inverse boundary value problems with partial data in unbounded domains.  
\newblock  Inverse Problems,   28 (2012),  085003.
 
\bibitem{S-J} D. Saint-James.
 \newblock Etude du champ critique $HC3$ dans une g\'eom\'etrie cylindrique.
 \newblock Physics Letters, (15)(1), 13-15, (1965).

\bibitem{SoSo} S. Soojin Son. 
\newblock Spectral Problems on Triangles and Discs:
Extremizers and Ground States. 
 PhD thesis. 2014. \\ url:
https://www.ideals.illinois.edu/items/49400.


\bibitem{Tat} D. Tataru.
\newblock  Unique continuation for pde's. \newblock The IMA Volumes in Mathematics and its Applications $\mathbf{137}$, (2003), 239-255.










\bibitem{Te2}N. Temme. 
\newblock Asymptotic method for integrals.
\newblock Series in Analysis, Vol. 6, World Scientific Publishing, (2015).
  
\end{thebibliography}
\end{document}